\documentclass[a4paper,cjk,11pt]{article}
\usepackage{bm}           
\usepackage{amsthm}       
\usepackage{amsfonts}     
\usepackage{color}        
\usepackage{graphicx}     
\usepackage{float}        
\usepackage{amsmath}      
\usepackage{amssymb}
\usepackage[mathscr]{euscript}
\usepackage{dsfont}

\usepackage[paper=a4paper,top=2cm,left=2cm,right=2cm,
    foot=1cm,bottom=3cm]{geometry}

\begin{document}

\title{Quaternion Gaussian Probes for Tensor Frobenius-Norm Estimation}
\author{%
  Yannan Chen\footnote{School of Mathematical Sciences, South China Normal University, Guangzhou 510631, China. (\texttt{email:~ynchen@scnu.edu.cn}) This author was supported by the National Natural Science Foundation of China (12671423, 12471351).} \and
  Hailin Sun\footnote{Key Laboratory of NSLSCS, Ministry of Education; Jiangsu International Joint Laboratory of BDMCA; School of Mathematical Sciences, Nanjing Normal University, Nanjing 210023, China.  (\texttt{email:~hlsun@njnu.edu.cn})} \and
  Min Xi\footnote{School of Mathematics and Statistics, Guangdong University of Foreign Studies, Guangzhou, 510006, China. (\texttt{email:~mxi@gdufs.edu.cn}) This author was supported by the Guangdong Province Higher Education Foundation (No. 2021ZDZX1071). }}
\date{\today}
\maketitle

\newcommand{\REAL}{\mathds{R}}
\newcommand{\COPX}{\mathds{C}}
\newcommand{\QUAN}{\mathds{H}}
\newcommand{\av}{\mathbf{a}}
\newcommand{\bv}{\mathbf{b}}
\newcommand{\cv}{\mathbf{c}}
\newcommand{\dv}{\mathbf{d}}
\newcommand{\qv}{\mathbf{q}}
\newcommand{\uv}{\mathbf{u}}
\newcommand{\vv}{\mathbf{v}}
\newcommand{\wv}{\mathbf{w}}
\newcommand{\xv}{\mathbf{x}}
\newcommand{\yv}{\mathbf{y}}
\newcommand{\zv}{\mathbf{z}}
\newcommand{\alphav}{{\bm\alpha}}
\newcommand{\Am}{\mathbf{A}}
\newcommand{\Bm}{\mathbf{B}}
\newcommand{\Cm}{\mathbf{C}}
\newcommand{\Id}{\mathbf{I}}
\newcommand{\Qm}{\mathbf{Q}}
\newcommand{\Sm}{\mathbf{S}}
\newcommand{\Um}{\mathbf{U}}
\newcommand{\Vm}{\mathbf{V}}
\newcommand{\Wm}{\mathbf{W}}
\newcommand{\Xm}{\mathbf{X}}
\newcommand{\Sigmam}{\mathbf{\Sigma}}
\newcommand{\At}{\mathcal{A}}
\newcommand{\Bt}{\mathcal{B}}
\newcommand{\Ct}{\mathcal{C}}
\newcommand{\St}{\mathcal{S}}

\newcommand{\vr}[1]{\vec{\boldsymbol{#1}}}
\newcommand{\qn}[1]{\boldsymbol{#1}}

\newcommand{\rank}{\mathrm{rank}}
\newcommand{\Prob}{\mathbb{P}}
\newcommand{\Pro}[1]{\mathbb{P}\left\{#1\right\}}
\newcommand{\Mean}{\mathbb{E}}
\newcommand{\Expe}{\mathbb{E}}
\newcommand{\Var}{\mathrm{Var}}
\newcommand{\diff}{\mathrm{d}}
\newcommand{\ii}{\mathbf{i}}
\newcommand{\jj}{\mathbf{j}}
\newcommand{\kk}{\mathbf{k}}
\newtheorem{Thm}{Theorem}[section]
\newtheorem{Def}[Thm]{Definition}
\newtheorem{Prop}[Thm]{Proposition}
\newtheorem{Lem}[Thm]{Lemma}
\newtheorem{Cor}[Thm]{Corollary}
\newtheorem{Exe}[Thm]{Exercise}
\newcounter{algorithm}
\renewcommand{\thealgorithm}{\arabic{algorithm}}

\numberwithin{equation}{section}
\newcommand{\argmin}{\mathop{\arg\,\min}}

\begin{abstract}
Large-scale tensors arising in science and engineering are often too large to form or store explicitly. Their multilinear contractions, however, can frequently be evaluated efficiently, which makes randomized Frobenius-norm estimation a practical alternative to exact computation. 
In this paper, we study estimation of the squared Frobenius norm of a real order-$d$ tensor from such contractions using independent real, complex, or quaternion Gaussian probes. For a rank-one tensor, one normalized sample is a product of $d-1$ independent Gamma variables. This gives exact relative-variance factors $3^{d-1}-1$, $2^{d-1}-1$, and $(3/2)^{d-1}-1$, respectively. Under the same budget of underlying real Gaussian vectors, the quaternion probe has smaller rank-one variance than four independent real probes for $d\ge3$; a real-channel representation identifies the gain as modewise pooling. For orthogonal CP tensors, the exact relative variance equals the corresponding rank-one factor divided by the participation-ratio effective rank and the number of samples. For arbitrary real tensors, we prove sharp uniform moment bounds, attained at rank one, and obtain Chebyshev guarantees. A classical block-median construction then gives two-sided exponential confidence for the modified estimator. Numerical examples test the rank-one and effective-rank formulas. An implicit infinite-series kernel tensor illustrates contraction-based norm evaluation; EEG maps provide a qualitative spatial interpolation illustration.

\medskip\noindent\textbf{Key words.} tensor Frobenius norm, randomized norm estimation, quaternion Gaussian probe, Chebyshev bound, orthogonal CP tensor, effective rank

\medskip\noindent\textbf{MSC (2020).} 15A69, 60B20, 65F35
\end{abstract}

\newpage

\section{Introduction}\label{intro}

The Frobenius norm is a natural measure of the size of a tensor and of
the error in its approximation. For an order-$d$ tensor
$\At\in\REAL^{n_1\times\cdots\times n_d}$, however, storing all
$\prod_{j=1}^d n_j$ entries may be infeasible even when the tensor can
be accessed through its factors or multilinear contractions
\cite{KoldaBader09,TensBK}. The need to estimate its norm without
materializing it arises in several concrete settings. 

First, suppose that an adaptive low-rank method produces an
approximation $\widetilde{\At}_r$. Its stopping decision may require
checking $\|\At-\widetilde{\At}_r\|_F\leq\varepsilon\|\At\|_F$ with $\varepsilon\in(0,1)$,
although the residual $\mathcal R_r=\At-\widetilde{\At}_r$ is too
large to assemble. Canonical polyadic (CP), Tucker, and tensor-train representations make structured
contractions available, and adaptive tensor methods explicitly require
error control \cite{BK08,Oseledets11,BachmayrDahmen15,CheWeiYan2026,ChangCuiHanWang2026}. A direct method for the triple decomposition provides another structured representation of third-order tensors \cite{ChenQi2026}. Random probes
can therefore provide a posterior norm estimate whenever the residual
contractions are computable; a probabilistic stopping certificate also
needs an appropriate confidence bound.

Second, empirical moment tensors in latent-variable learning can be
formed implicitly from a stream $\av_1,\ldots,\av_T\in\REAL^n$:
\[
 \mathcal M_{d,T}=\frac{1}{T}\sum_{t=1}^T\av_t^{\circ d}.
\]
Moment-based tensor methods are well established
\cite{AnandkumarEtAl14}, but $\mathcal M_{d,T}$ has $n^d$ entries.
For fixed modewise probes, each multilinear contraction can instead
be updated from the incoming rank-one term. Monitoring
$\|\mathcal M_{d,T}\|_F$ or the norm of a change in the moment tensor
is one possible use of such sketches; its statistical interpretation
depends on the learning problem.

Third, the order-$d$ derivative $\nabla^d f(\xv)$ of a smooth scalar
function $f(\xv)$ is usually too large to store. Repeated directional
differentiation and automatic differentiation can evaluate its
contraction with $d-1$ prescribed directions, leaving a vector in the
first mode \cite{BaydinEtAl18}. Estimating its Frobenius norm would
then give a matrix-free diagnostic of high-order curvature. 

These examples motivate the same oracle question: given access to
\[
 (\cdots(\At\times_d\xv_d)\cdots)\times_2\xv_2, 
\]
how accurately can one estimate $\|\At\|_F^2$ from a modest number
of independent contractions?

For matrices, this question is related to stochastic trace estimation:
with $\Bm=\Am^T\Am$, one has $\|\Am\|_F^2=\operatorname{tr}\Bm$ and
$\mathbb E\|\Am\xv\|_2^2=\operatorname{tr}\Bm$
for an isotropic random vector $\xv$.
Hutchinson's estimator and subsequent high-probability analyses
developed this matrix-vector framework
\cite{H1990,AT2011,CK22}. For higher-order tensors, products of
independent modewise probes preserve access to contractions, but their
moments depend strongly on the tensor order. Bujanovi\'c and
Kressner studied norm and trace estimation with random rank-one
vectors \cite{BK21}; Meyer and Avron derived exact variance formulas
for Kronecker trace estimation, including real and complex Gaussian
probes \cite{MA26}. Related tensor
Bernstein inequalities concern random tensors under the Einstein
product \cite{LQT20}, while general Gaussian-chaos estimates describe
the tails of polynomial functions of Gaussian variables \cite{La06}. The choice of random probes also affects low-rank approximation \cite{SaibabaMiedlar2025}.

For real and complex probes, one sample of our estimator is
a Kronecker-probe trace sample for the positive
semidefinite matrix $\Am_{(1)}^T\Am_{(1)}$. Our focus is the
corresponding quaternion estimator for a real tensor, with a fixed
order of contractions. 
For $k$ independent probe tuples, we study the sample mean
\[
 \mathrm{Est}_k^G
 :=\frac1k\sum_{i=1}^k
 \left\|(\cdots(\At\times_d\xv_d^i)\cdots)\times_2\xv_2^i\right\|_2^2,
 \qquad \xv_j^i\in\mathds F^{n_j}, ~j=2,\dots,d,
 \qquad \mathds F\in\{\REAL,\COPX,\QUAN\}.
\]
The probes are independent across modes and trials and are scaled so that $\mathbb E(\xv_j^i(\xv_j^i)^H)=\Id_{n_j}$. The real tensor coefficients and the fixed order of successive contractions make this one estimator definition for all three fields. The main results are as follows.
\begin{enumerate}
\item \textit{Exact rank-one law and a comparison at equal real Gaussian budget.}
We show that the estimator is unbiased and strongly consistent. For a nonzero rank-one tensor, a single sample divided by $\|\At\|_F^2$ is a product of $d-1$ independent normalized chi-square variables. Its exact moments yield the sharp real, complex, and quaternion relative-variance factors
\[
 3^{d-1}-1,\qquad 2^{d-1}-1,\qquad (3/2)^{d-1}-1.
\]
One quaternion probe uses four real Gaussian vectors per contracted mode. Its rank-one variance equals that of four independent real samples for $d=2$ and is strictly smaller for $d\ge3$. The gain arises from pooling the four channels within each mode.

\item \textit{Orthogonal CP structure and effective rank.}
For $\At=\sum_{r=1}^R\lambda_r\uv_{1r}\circ\cdots\circ\uv_{dr}$ with orthonormal factors in every mode, define
$r_{\mathrm{eff}}(\At)=(\sum_r\lambda_r^2)^2/\sum_r\lambda_r^4$.
We prove the exact identity
\[
 \frac{\operatorname{Var}(\mathrm{Est}_k^G)}{\|\At\|_F^4}
 =\frac{\varrho(2,\mathds F)^{d-1}-1}
        {k\,r_{\mathrm{eff}}(\At)},
\]
where $\varrho(2,\mathds F)$ equals $3$, $2$, and $3/2$ over $\REAL$, $\COPX$, and $\QUAN$, respectively.
Thus the energy distribution among orthogonal components determines the reduction beyond the rank-one variance. The effective rank can also be computed from any mode unfolding.

\item \textit{General-tensor moments and confidence.}
Without a rank or orthogonality assumption, we prove, for each integer $p\ge1$,
$\mathbb EY^p\le\varrho(p,\mathds F)^{d-1}\|\At\|_F^{2p}$.
Rank-one equality makes this a sharp worst-case bound. It yields a variance and a two-sided Chebyshev estimate for the ordinary mean. Because higher-order rank-one Gaussian products have a heavy upper tail, we use the classical block-median principle \cite{BubeckCesaBianchiLugosi2013} to turn the finite variance bound into a two-sided exponential-confidence guarantee for a robust estimator. 

\item \textit{Numerical checks and an implicit-series application.}
Rank-one simulations compare the three probes at equal real-randomness budget; a separate experiment uses genuinely orthogonal CP factors to test the $1/r_{\mathrm{eff}}$ law. For a kernel coefficient tensor given by an infinite sum of rank-one terms \cite{XY19,CFLY24}, we evaluate contractions without materializing its entries. Scalp-potential maps illustrate a related spatial interpolation workflow qualitatively. 
\end{enumerate}


The remainder of this paper is organized as follows. Section~\ref{Sec:EstMethod} defines the estimator and proves consistency. Section~\ref{Sec:Concentration} gives the rank-one and equal-budget results. Section~\ref{sec:ocp-main} treats orthogonal CP tensors, Section~\ref{sec:general-main} gives the uniform moment and concentration bounds, and Section~\ref{Sec:NumExp} studies the implicit-series example and the qualitative interpolation maps. Section~\ref{Sec:Conclude} concludes the paper.

\section{Quaternion-based Gaussian estimator}\label{Sec:EstMethod}

We begin with the quaternion algebra $\QUAN$ introduced by William Rowan Hamilton \cite{Hamilton1844}. A quaternion $\qn{q}=(q_0,\vr{q}) \in\QUAN$, with $\vr{q}=(q_1,q_2,q_3)^T$, is written as
\begin{equation*}
	\qn{q} = q_0 +q_1\mathbf{i}+q_2\mathbf{j}+q_3\mathbf{k} \qquad\text{ with }~
	\mathbf{i}^2=\mathbf{j}^2=\mathbf{k}^2=\mathbf{i}\mathbf{j}\mathbf{k}=-1,
\end{equation*}
where $q_0$ is the real part, $q_1\mathbf{i}$, $q_2\mathbf{j}$ and $q_3\mathbf{k}$ are imaginary parts, and $(1,\mathbf{i},\mathbf{j},\mathbf{k})$ is a basis. Quaternions follow the normal rules of algebra except for the commutative law for multiplication (i.e., $\qn{p}\qn{q}\ne \qn{q}\qn{p}$). Indeed, the multiplication of two quaternions $\qn{p}=(p_0,\vr{p})$ and $\qn{q}=(q_0,\vr{q})$ produces
\begin{equation}\label{quatMult}
  \qn{p}\qn{q} = (p_0q_0-\vr{p}\cdot\vr{q}, p_0\vr{q}+q_0\vr{p}+\vr{p}\times\vr{q}),
\end{equation}
where ``$\cdot$'' and ``$\times$'' stand for the scalar and vector products of vectors, respectively. Owing to $\vr{p}\times\vr{q}=-\vr{q}\times\vr{p}$, the commutative law for multiplication of quaternions is not applicable. Quaternion multiplication is associative and has a unit $(1,\vr{0})$. The conjugate of a quaternion $\qn{q}$ is $\overline{\qn{q}}:=q_0-q_1\mathbf{i}-q_2\mathbf{j}-q_3\mathbf{k}$. The modulus of a quaternion is defined as $|\qn{q}|:=\sqrt{\qn{q}\overline{\qn{q}}}$. 

We now turn to randomized estimation. The key idea behind randomized Frobenius-norm estimation for matrices is to probe the matrix with random vectors, thereby projecting it onto random directions, and then estimate the norm by averaging the squared norms of the resulting projections. Analogously, for a high-order tensor $$\At \in \mathbb{R}^{n_1\times n_2\times \cdots \times n_d},$$ we seek a random projection mechanism that retains its multilinear structure while enabling efficient norm estimation.

Let us review Gaussian random vectors with real and complex values \cite{MA26}. Let $\qv_0$ and $\qv_1$ be independent real random vectors, each following the standard multivariate normal distribution $\mathsf{N}(\mathbf{0},\Id_n)$. Define the following random vectors.
\begin{itemize}
  \item Common real-valued random vectors
  \begin{equation*}
    \xv^{\REAL} = \qv_0 \sim \mathsf{N}(\mathbf{0},\Id_n).
  \end{equation*}
  \item Complex-valued random vectors
  \begin{equation*}
    \xv^{\COPX} = \frac{1}{\sqrt{2}}\qv_0+\frac{\ii}{\sqrt{2}}\qv_1 \sim \mathsf{N}(\mathbf{0},\Id_n).
  \end{equation*}
\end{itemize}
Similarly, let $\qv_0, \qv_1, \qv_2, \qv_3\sim \mathsf{N}(\mathbf{0},\Id_n)$ be independent real random vectors. We define
\begin{itemize}
  \item Quaternion-valued random vectors
  \begin{equation}\label{quat-rv}
    \xv^{\QUAN} = \frac{1}{2}\qv_0+\frac{\ii}{2}\qv_1+\frac{\jj}{2}\qv_2+\frac{\kk}{2}\qv_3.
  \end{equation}
\end{itemize}

The following lemma verifies the normalization of the quaternion Gaussian vector $\xv^{\QUAN}$.

\begin{Lem}\label{Lem:3.1}
  The quaternion Gaussian vector $\xv^{\QUAN}$ defined in
  \eqref{quat-rv} is zero-mean and isotropic; namely,
  \begin{equation*}
    \Mean(\xv^{\QUAN})=\mathbf{0},
    \qquad
    \mathbf{\Sigma}(\xv^{\QUAN})
    :=\Mean\!\left[
      \xv^{\QUAN}(\xv^{\QUAN})^H
    \right]
    =\Id_n.
  \end{equation*}
\end{Lem}
\begin{proof}
The four real components in \eqref{quat-rv} have mean zero and
are independent across coordinates. Thus
$\mathbb E(x_s\overline{x_t})=0$ for $s\ne t$, while
$\mathbb E|x_s|^2=(1/4)\sum_{\ell=0}^{3}\mathbb E(q_{\ell,s}^2)=1$.
This proves the covariance identity. Moreover, a quaternion unitary
map preserves the Euclidean norm on the underlying real space
$\REAL^{4n}$ and acts there as a real orthogonal map. The vector in
\eqref{quat-rv} is an isotropic real Gaussian on that space, scaled
by $1/2$, so its law is invariant under such maps. In particular,
its coordinates in any quaternion orthonormal basis are independent
standard quaternion Gaussian variables.
\end{proof}

This is the standard covariance normalization for a quaternion Gaussian
vector constructed from four independent real Gaussian components; see
\cite{TookMandic2011}. Accordingly, we call $\xv^{\QUAN}$ a standard quaternion Gaussian vector and write $\xv^{\QUAN}\sim\mathsf N(\mathbf0,\Id_n)$.

Next, we consider projecting a high-order tensor $\At$ along its second to $d$th modes using independent standard Gaussian random vectors
\begin{equation*}
  \xv_2\sim \mathsf{N}(\mathbf{0},\Id_{n_2}),\quad\dots,\quad\xv_d\sim \mathsf{N}(\mathbf{0},\Id_{n_d}),
\end{equation*}
where the probes $\xv_2,\dots,\xv_d$ are real-, complex-, or quaternion-valued. Successive modewise contractions produce the random vector $(\cdots(\At\times_d\xv_d)\cdots)\times_2\xv_2\in\mathds F^{n_1}$ over $\mathds F\in\{\REAL,\COPX,\QUAN\}$, whose squared norm
\begin{equation*}
  \|(\cdots(\At \times_d\xv_d) \cdots )\times_2\xv_2\|_2^2 = \|\Am_{(1)}(\xv_d\otimes \cdots\otimes \xv_2)\|_2^2
\end{equation*}
serves as a Monte Carlo sample for the squared tensor Frobenius norm $\|\At\|_F^2$, where $\Am_{(1)}\in\REAL^{n_1\times n_2 \cdots n_d}$ denotes the mode-$1$ unfolding matrix of $\At$, and ``$\otimes$'' stands for the Kronecker product.
For quaternion probes, the product convention in this identity is
fixed by the successive contractions: its entry indexed by
$(i_d,\ldots,i_2)$ is $(x_d)_{i_d}\cdots(x_2)_{i_2}$, in precisely this order.

Given a positive integer $k$, we generate $d-1$ independent random matrices
\begin{equation*}
  \Xm_2=(\xv_2^1,\dots,\xv_2^k),\quad \dots, \quad \Xm_d=(\xv_d^1,\dots,\xv_d^k),
\end{equation*}
whose entries are mutually independent standard Gaussian variables over $\mathds F\in\{\REAL,\COPX,\QUAN\}$. Averaging $k\ge1$ independent trials gives the estimator
\begin{equation}\label{tensor-estimator}
  \mathrm{Est}^G_k := \frac{1}{k} \sum_{i=1}^k \left\| (\cdots(\At \times_d \xv_d^i) \cdots) \times_2 \xv_2^i \right\|_2^2
  = \frac{1}{k}\left\|\Am_{(1)}(\Xm_d\odot\cdots\odot\Xm_2)\right\|_F^2,
\end{equation}
where $\mathrm{Est}^G_k$ approximates $\|\mathcal{A}\|_F^2$ and ``$\odot$'' stands for the Khatri-Rao (columnwise Kronecker) product. The purpose of performing multiple Monte Carlo samplings is to reduce the variance of the estimator and improve the accuracy.

The estimator in \eqref{tensor-estimator} does not require the entries of $\At$ to be explicitly available. It only requires the multilinear contractions
\begin{equation*}
  (\cdots(\At\times_d\xv_d^i)\cdots)\times_2\xv_2^i,
  \qquad i=1,\ldots,k.
\end{equation*}
Consequently, it applies to the implicit residual, streaming moment, and derivative tensors discussed in the introduction whenever these contractions can be evaluated through the underlying representation or computational oracle. The $k$ samples can be computed independently and hence in parallel. For tensors represented as additive streams, the corresponding sketches can also be updated term by term without storing the accumulated tensor. The sample size $k$ therefore provides a direct tradeoff between computational cost and estimation accuracy.

\begin{Prop}[Unbiasedness]\label{Thm:unbiased}
  Let $\At\in\REAL^{n_1\times\cdots\times n_d}$, and let $\mathrm{Est}_k^G$ be the estimator defined in \eqref{tensor-estimator}. For each sample $i=1,\ldots,k$ and each mode $j=2,\ldots,d$, let $\xv_j^i\in\mathds F^{n_j}$ be a standard Gaussian vector over $\mathds F\in\{\REAL,\COPX,\QUAN\}$. Assume that these vectors are mutually independent and normalized so that $\mathbb E\!\left(\xv_j^i(\xv_j^i)^H\right)=\Id_{n_j}$. Then $\mathrm{Est}_k^G$ is an unbiased estimator of the squared Frobenius norm of $\At$; that is,
\begin{equation*}
 \mathbb E(\mathrm{Est}_k^G)=\|\At\|_F^2.
\end{equation*}
\end{Prop}
\begin{proof}
Because the $k$ Monte Carlo samples are identically distributed, it is enough to consider one of them. We suppress the sample index and write
\begin{equation*}
 Y:=\left\|(\cdots(\At\times_d\xv_d)\cdots)\times_2\xv_2\right\|_2^2.
\end{equation*}
We first record the effect of a single random contraction. Let $\mathcal B$ be any tensor over $\mathds F$ that is independent of an isotropic vector $\xv\in\mathds F^n$, and suppose that the contraction is performed along a mode of length $n$. If $\bv_\alpha^T$ denotes a row of the corresponding unfolding of $\mathcal B$, then
\begin{align*}
 \mathbb E_{\xv}\|\mathcal B\times\xv\|_F^2
 &=\sum_\alpha
   \mathbb E_{\xv}\left|\bv_\alpha^T\xv\right|^2\\
 &=\sum_\alpha\sum_{s=1}^{n}\sum_{t=1}^{n}
   b_{\alpha s}\,\mathbb E(x_s\overline{x_t})\,
   \overline{b_{\alpha t}}\\
 &=\sum_\alpha\sum_{s=1}^{n}|b_{\alpha s}|^2
 =\|\mathcal B\|_F^2.
\end{align*}
Here $\mathbb E(x_s\overline{x_t})=\delta_{st}$ follows from isotropy. The quantity $\delta_{st}$ is real, so the calculation remains valid in the quaternion case despite the noncommutativity of quaternion multiplication. Thus, contraction with one independent isotropic probe preserves the squared Frobenius norm in expectation.

Apply the identity in the reverse order of the displayed
contractions. First condition on $\xv_3,\ldots,\xv_d$ and integrate
the last-applied probe $\xv_2$; the tensor
$(\cdots(\At\times_d\xv_d)\cdots)\times_3\xv_3$ is then fixed and
independent of $\xv_2$. Its expected squared norm after contraction
with $\xv_2$ is its squared Frobenius norm. Next integrate $\xv_3$,
then $\xv_4$, and so on through $\xv_d$. At every step the remaining
tensor is independent of the probe being integrated. This order also
preserves the stated quaternion multiplication order, without
commuting any quaternion factors. Thus
\begin{equation*}
 \mathbb EY=\|\At\|_F^2.
\end{equation*}
Finally, $\mathrm{Est}_k^G$ is the average of $k$ copies of $Y$, and hence
\begin{equation*}
 \mathbb E(\mathrm{Est}_k^G)
 =\frac1k\sum_{i=1}^k\mathbb EY_i
 =\|\At\|_F^2. \qedhere
\end{equation*}
\end{proof}

Unbiasedness is a standard consequence of isotropy. Related results for real- and complex-valued Kronecker probes appear in \cite{BK21,MA26}, while the quaternion case follows from the same argument using Lemma~\ref{Lem:3.1}. The proposition records a basic property of the unified estimator.

\begin{Cor}[Almost sure consistency]\label{cor:strong-consistency} 
  By the strong law of large numbers, $\mathrm{Est}_k^G\to\|\At\|_F^2$ almost surely as $k\to\infty$, for every fixed real tensor $\At$ and each choice of $\mathds F$.
\end{Cor}

This consistency statement explains in which precise sense increasing
$k$ makes randomized norm estimation reliable. The finite-sample
bounds in the next sections quantify the required sample size.

\section{Rank-one tensors: exact variance and real Gaussian pooling}\label{Sec:Concentration}

We first consider the rank-one case, which identifies the sharp constants for the sample mean.

\begin{Def}\label{def:field-degree} Define
\begin{equation}\label{field-degree}
 \beta_{\mathds F}:=
 \begin{cases}
  1,&\mathds F=\REAL,\\
  2,&\mathds F=\COPX,\\
  4,&\mathds F=\QUAN.
 \end{cases}
\end{equation}
The parameter $\beta_{\mathds F}$ counts independent real Gaussian
components in one scalar probe.
\end{Def}

The following lemma combines the standard chi-square and Gamma moment identities \cite{CB24} with the real-component representation of quaternion Gaussian variables.

\begin{Lem}[Scalar Gaussian moments]\label{lem:scalar-gamma}
  Let $g\sim\mathsf N_{\mathds F}(0,1)$. Then
\begin{equation*}
 |g|^2\sim\operatorname{Gamma}\left(\frac{\beta_{\mathds F}}{2},\frac{2}{\beta_{\mathds F}}\right),
\end{equation*}
where the second parameter is the scale. For every $p>-\beta_{\mathds F}/2$,
\begin{equation}\label{scalar-moment}
 \mathbb E|g|^{2p}
 =\left(\frac{2}{\beta_{\mathds F}}\right)^p
 \frac{\Gamma(\beta_{\mathds F}/2+p)}{\Gamma(\beta_{\mathds F}/2)}.
\end{equation}
In particular, $\mathbb E|g|^2=1$ and $\mathbb E|g|^4=1+2/\beta_{\mathds F}$.
\end{Lem}
\begin{proof}
  Write $g=\beta_{\mathds F}^{-1/2}\sum_{\ell=1}^{\beta_{\mathds F}}q_\ell e_\ell$, where the $q_\ell$ are independent real $\mathsf N(0,1)$ variables and the $e_\ell$ are the corresponding real, complex, or quaternion basis elements. Then
\begin{equation*}
 |g|^2=\frac1{\beta_{\mathds F}}\sum_{\ell=1}^{\beta_{\mathds F}}q_\ell^2
 \sim\frac1{\beta_{\mathds F}}\chi_{\beta_{\mathds F}}^2.
\end{equation*}
The Gamma moment formula \cite{CB24} gives \eqref{scalar-moment}.
\end{proof}

The Gamma representation is classical. It is recorded in a unified normalization to expose the scalar fourth moments $3$, $2$, and $3/2$.

\begin{Thm}[Exact rank-one distribution and variance]\label{thm:rank-one-exact}
  Let $\At=\av_1\circ\cdots\circ\av_d$ be a nonzero real rank-one tensor and let
\begin{equation*}
 Y=\left\|(\cdots(\At\times_d\xv_d)\cdots)\times_2\xv_2\right\|_2^2,
\end{equation*}
where the $\xv_j\sim\mathsf N_{\mathds F}(\mathbf0,\Id_{n_j})$ are mutually independent. Then
\begin{equation}\label{rank-one-distribution}
 \frac{Y}{\|\At\|_F^2}\overset{\mathrm d}=
 \prod_{j=2}^d Z_j,\qquad
 Z_j\overset{\mathrm{i.i.d.}}\sim
 \operatorname{Gamma}\left(\frac{\beta_{\mathds F}}{2},\frac{2}{\beta_{\mathds F}}\right).
\end{equation}
Consequently, for every $p>-\beta_{\mathds F}/2$,
\begin{equation}\label{rank-one-moments}
 \mathbb E\left(\frac{Y}{\|\At\|_F^2}\right)^p
 =\left[
 \left(\frac{2}{\beta_{\mathds F}}\right)^p
 \frac{\Gamma(\beta_{\mathds F}/2+p)}{\Gamma(\beta_{\mathds F}/2)}
 \right]^{d-1},
\end{equation}
and
\begin{equation}\label{rank-one-variance}
 \operatorname{Var}(\mathrm{Est}_k^G)
 =\frac{(1+2/\beta_{\mathds F})^{d-1}-1}{k}\|\At\|_F^4.
\end{equation}
\end{Thm}

\begin{proof}
  Multilinearity and multiplicativity of the modulus give
\begin{equation*}
 Y=\|\av_1\|_2^2\prod_{j=2}^d
 |\langle\av_j,\xv_j\rangle|^2.
\end{equation*}
Rotational invariance implies
\begin{equation*}
 g_j:=\frac{\langle\av_j,\xv_j\rangle}{\|\av_j\|_2}
 \sim\mathsf N_{\mathds F}(0,1).
\end{equation*}
Since $\|\At\|_F^2=\prod_{j=1}^d\|\av_j\|_2^2$, Lemma~\ref{lem:scalar-gamma} means 
\begin{equation*}
  \frac{Y}{\|\At\|_F^2}=\prod_{j=2}^{d}\frac{|\langle\av_j,\xv_j\rangle|^2}{\|\av_j\|_2^2}=\prod_{j=2}^{d}|g_j|^2  
  \overset{\mathrm d}=
 \prod_{j=2}^d Z_j,
\end{equation*}
which proves \eqref{rank-one-distribution}. The independence yields \eqref{rank-one-moments}. 

Setting $p=1$ and $p=2$ in \eqref{rank-one-moments}, respectively,
gives
\begin{equation*}
  \mathbb EY=\|\At\|_F^2,
  \qquad
  \mathbb EY^2
  =
  \left(1+\frac{2}{\beta_{\mathds F}}\right)^{d-1}\|\At\|_F^4.
\end{equation*}
Consequently,
\begin{equation*}
  \operatorname{Var}(Y)
  =
  \left[
    \left(1+\frac{2}{\beta_{\mathds F}}\right)^{d-1}-1
  \right]\|\At\|_F^4.
\end{equation*}
Since $\mathrm{Est}_k^G$ is the average of $k$ independent copies of
$Y$, its variance is $\operatorname{Var}(Y)/k$, which proves
\eqref{rank-one-variance}.
\end{proof}

The exact variance shows that its exponential dependence on $d-1$ is intrinsic. Motivated by the rank-one formula, define
\begin{equation}\label{rho-definition}
 \varrho(p,\mathds F):=
 \left(\frac{2}{\beta_{\mathds F}}\right)^p
 \frac{\Gamma(p+\beta_{\mathds F}/2)}
 {\Gamma(\beta_{\mathds F}/2)}.
\end{equation}
For integer $p\ge1$, this equals $(2p-1)!!$, $p!$, and $(p+1)!/2^p$ over $\REAL$, $\COPX$, and $\QUAN$, respectively.

\subsection{Equal-budget comparison through real Gaussian pooling}\label{sec:equal-budget}

The smaller fourth moment of a quaternion Gaussian probe raises a
natural question: does its variance reduction reflect a genuinely
noncommutative effect, or does it arise from the pooling of several
real Gaussian components?

As in \eqref{quat-rv}, a standard quaternion Gaussian vector admits the representation
\begin{equation}\label{quaternion-real-components}
  \xv
  =
  \frac{1}{2}
  \left(
    \qv_{0}
    +\qv_{1}\ii
    +\qv_{2}\jj
    +\qv_{3}\kk
  \right),
\end{equation}
where
\(
\qv_{0},\ldots,\qv_{3}
\)
are independent standard real Gaussian vectors. Thus, one quaternion
probe uses four independent real Gaussian components in each contracted
mode. In the matrix case \(d=2\), this decomposition gives
\begin{equation}\label{matrix-four-real}
  \|\Am\xv\|_2^2
  =
  \frac{1}{4}
  \sum_{\ell=0}^{3}\|\Am\qv_{\ell}\|_2^2.
\end{equation}
Consequently, one quaternion matrix probe is statistically equivalent
to the average of four independent real Gaussian probes. For tensors
of order \(d\ge3\), however, the four real components are pooled within
each mode before the resulting modewise energies are multiplied. This
difference leads to a strict variance improvement.

\begin{Thm}[Equal-real-randomness comparison for rank-one tensors]
\label{thm:equal-budget}
Let \(\At\ne0\) be a rank-one tensor of order \(d\). Denote by \(Y_{\QUAN}\) one quaternion Gaussian sample and by
\[
  \overline{Y}_{4,\REAL}
  :=
  \frac{1}{4}\sum_{\ell=0}^{3}Y_{\REAL}^{(\ell)}
\]
the average of four independent real Gaussian tensor samples. Both
estimators use \(4(d-1)\) independent real Gaussian probe vectors. Their
relative variances are
\begin{equation}\label{equal-budget-variances}
  \frac{\operatorname{Var}(Y_{\QUAN})}{\|\At\|_F^4}
  =
  \left(\frac{3}{2}\right)^{d-1}-1,
  \qquad
  \frac{\operatorname{Var}(\overline{Y}_{4,\REAL})}{\|\At\|_F^4}
  =
  \frac{3^{d-1}-1}{4}.
\end{equation}
The two variances are equal when \(d=2\), whereas
\begin{equation}\label{equal-budget-strict}
  \left(\frac{3}{2}\right)^{d-1}-1
  <
  \frac{3^{d-1}-1}{4}
\end{equation}
for every \(d\ge3\).
\end{Thm}

\begin{proof}
The exact rank-one variance formula in
Theorem~\ref{thm:rank-one-exact} gives
\[
  \frac{\operatorname{Var}(Y_{\QUAN})}{\|\At\|_F^4}
  =
  \left(\frac{3}{2}\right)^{d-1}-1.
\]
Each real Gaussian sample has relative variance \(3^{d-1}-1\). Since
\(\overline{Y}_{4,\REAL}\) is the average of four independent samples,
\[
  \frac{\operatorname{Var}(\overline{Y}_{4,\REAL})}{\|\At\|_F^4}
  =
  \frac{3^{d-1}-1}{4}.
\]
This proves \eqref{equal-budget-variances}.

To compare the two quantities, define
\[
  h_d:=3^{d-1}-4\left(\frac{3}{2}\right)^{d-1}+3.
\]
Then \(h_2=0\), and
\[
  h_{d+1}
  =
  3h_d+6\left[\left(\frac{3}{2}\right)^{d-1}-1\right].
\]
Hence \(h_d>0\) for every \(d\ge3\), which is equivalent to
\eqref{equal-budget-strict}.
\end{proof}

The comparison uses the same number of underlying real Gaussian
draws per mode. It establishes a variance advantage in the rank-one
case; the costs of contractions and quaternion arithmetic must be
considered separately.

For a rank-one tensor
\(
\At=\av_1\circ\cdots\circ\av_d
\),
the quaternion estimator has the entirely real distributional
representation
\begin{equation}\label{rank-one-real-pooling}
  \frac{Y_{\QUAN}}{\|\At\|_F^2}
  \overset{\mathrm d}{=}
  \prod_{j=2}^{d}
  \left(
    \frac{1}{4}\sum_{a=0}^{3}g_{j,a}^2
  \right),
  \qquad
  g_{j,a}\stackrel{\mathrm{iid}}{\sim}\mathsf N(0,1).
\end{equation}
Therefore, the strict gain for \(d\ge3\) is explained by the pooling
of four real Gaussian components within each mode before the modewise
energies are multiplied. It does not require noncommutativity in the
rank-one setting. More generally, suppose mode $j$ pools an
integer number $\beta_j\ge1$ of independent real Gaussian squares.
Its normalized scalar energy has second moment $1+2/\beta_j$, so
the variance is minimized by minimizing
$\prod_{j=2}^d(1+2/\beta_j)$ for a fixed total
$B=\sum_{j=2}^d\beta_j$. The function
$f(b)=\log(1+2/b)$ is strictly convex for $b>0$.
For allocations, Jensen's inequality gives
$\beta_j=B/(d-1)$ for every $j$. 

We illustrate Theorems~\ref{thm:rank-one-exact} and \ref{thm:equal-budget} by estimating the squared
Frobenius norm of real rank-one tensors with real, complex, and
quaternion probes constructed from the same underlying real random
numbers. For each order $d\in\{2,\ldots,8\}$, we generate independent
Gaussian vectors $\av_j\in\REAL^{1000}$ and normalize them to unit
length, $j=1,\ldots,d$. Thus
$\At=\av_1\circ\cdots\circ\av_d$ has
$\|\At\|_F^2=1$.

In each of $500$ independent repetitions, we draw $d-1$ arrays
$G_j\in\REAL^{1000\times200}$ with independent standard normal
entries. For each contracted mode $j$, the $200$ columns of $G_j$
serve directly as real probes. Pairing consecutive columns and
dividing by $\sqrt2$ produces $100$ standard complex probes; grouping
four consecutive columns and dividing by $2$ produces $50$ standard
quaternion probes as in \eqref{quat-rv}. Hence the three methods use
the same $200\cdot1000(d-1)$ real Gaussian numbers per repetition,
although their numbers of tensor samples are respectively
$k_{\REAL}=200$, $k_{\COPX}=100$, and $k_{\QUAN}=50$.
The methods are coupled through $G_j$, but each method separately has
independent standard probes across samples and modes, as required by
Theorem~\ref{thm:rank-one-exact}.

Writing $\beta_{\mathds F}\in\{1,2,4\}$ as in
\eqref{field-degree}, the exact relative variance predicted by
\eqref{rank-one-variance} is
\begin{equation}\label{equal-budget-three-fields}
 \frac{\operatorname{Var}(\mathrm{Est}_{k_{\mathds F}}^G)}{\|\At\|_F^4}
 =\frac{\beta_{\mathds F}}{200}
 \left[\left(1+\frac{2}{\beta_{\mathds F}}\right)^{d-1}-1\right]
 =\begin{cases}
 (3^{d-1}-1)/200,&\mathds F=\REAL,\\
 (2^{d-1}-1)/100,&\mathds F=\COPX,\\
 ((3/2)^{d-1}-1)/50,&\mathds F=\QUAN.
 \end{cases}
\end{equation}
For $d=2$, all three values equal $0.01$. For every $d\ge3$,
the formula gives a strict real--complex--quaternion ordering as $\beta_{\mathds F}$ varies. In particular, the
$200$ real samples are exactly $50$ four-real averages, so their
comparison with the $50$ quaternion samples tests
Theorem~\ref{thm:equal-budget} at the same total randomness budget.

\begin{figure}[tb]
\centering
\includegraphics[width=0.82\textwidth]{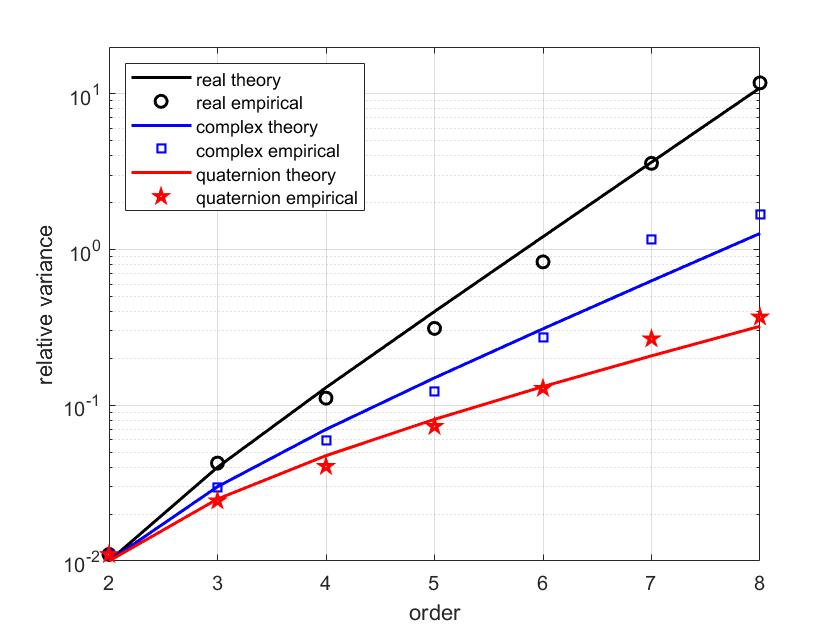}
\caption{Relative variance for rank-one tensors under the same budget
of real Gaussian numbers. Solid curves are the exact values in
\eqref{equal-budget-three-fields}; markers are sample variances across
$500$ independent repetitions. }
\label{fig:rankone-equal-budget}
\end{figure}

Figure~\ref{fig:rankone-equal-budget} compares sample variances
with the exact curves. The theoretical values coincide at $d=2$ and
have the real--complex--quaternion ordering for the displayed higher
orders. The markers broadly follow these values, with larger
departures at higher order. Since each marker uses only $500$
repetitions and one rank-one sample is a product of $d-1$ Gamma
variables, the empirical variance can fluctuate substantially.
The plot checks the predicted ordering and scale; it does not give
a precise estimate of the high-order variance constants.

\section{Orthogonal CP tensors and effective rank}\label{sec:ocp-main}

The rank-one result extends to an exact structure-dependent formula for
orthogonal CP tensors. Consider
\begin{equation}\label{orthogonal-cp}
 \At=\sum_{r=1}^{R}\lambda_r\,
 \uv_{1r}\circ\uv_{2r}\circ\cdots\circ\uv_{dr},
 \qquad \lambda_r\ne0,\qquad
 \uv_{jr}^{T}\uv_{js}=\delta_{rs},
\end{equation}
for all $j=1,\ldots,d$ and $r,s=1,\ldots,R$, where the symbol $\delta_{rs}$ denotes the Kronecker delta.
To describe how the Frobenius energy is distributed among the
orthogonal CP components, we make the following definition.

\begin{Def}[Participation-ratio effective rank]\label{def:ocp-effective-rank}
Let $\At$ have the orthogonal CP decomposition
\eqref{orthogonal-cp}. Set $p_r:=\lambda_r^2/\|\At\|_F^2$,
so $p_r>0$ and $\sum_{r=1}^R p_r=1$. The participation-ratio
effective rank of $\At$ is
\begin{equation}\label{effective-rank}
 r_{\mathrm{eff}}(\At)
 :=\frac{1}{\sum_{r=1}^R p_r^2}
 =\frac{\left(\sum_{r=1}^R\lambda_r^2\right)^2}
 {\sum_{r=1}^R\lambda_r^4}.
\end{equation}
\end{Def}

This is the participation ratio of the normalized component
energies $p_r$ \cite{RecanatesiEtAl2022}.
Formula \eqref{effective-rank} measures the effective number of orthogonal CP components contributing to the tensor.

\begin{Prop}[Basic properties of the effective rank]
\label{prop:effective-rank-properties}
Under the assumptions of Definition~\ref{def:ocp-effective-rank},
the following hold.
\begin{enumerate}
\item \emph{Range and equality.}
$1\le r_{\mathrm{eff}}(\At)\le R$.
The lower equality holds exactly when $R=1$; the upper equality
holds exactly when $|\lambda_1|=\cdots=|\lambda_R|$.
\item \emph{Mode-unfolding formula.}
For every $j=1,\ldots,d$, define
$G_j:=\Am_{(j)}\Am_{(j)}^T$. Then
\begin{equation}\label{effective-rank-unfolding}
 r_{\mathrm{eff}}(\At)
 =\frac{(\mathrm{trace}(G_j))^2}{\mathrm{trace}(G_j^2)}
 =\frac{\|\Am_{(j)}\|_F^4}{\|\Am_{(j)}\Am_{(j)}^T\|_F^2},
 \qquad j=1,2,\dots,d.
\end{equation}
Thus the same effective rank is obtained from every mode unfolding.
\item \emph{Relation to the coefficient of variation (CV).}
The mean of the weights is $\bar p:=\frac{1}{R}\sum_{r=1}^R p_r=1/R$.
Their coefficient of variation is
\begin{equation}\label{effective-cv-definition}
 \mathrm{CV}_p
 :=\frac{\sqrt{\frac{1}{R}\sum\limits_{r=1}^R(p_r-\bar p)^2}}{\bar p}
 =\sqrt{R\sum_{r=1}^R p_r^2-1}.
\end{equation}
Consequently,
\begin{equation}\label{effective-cv-identity}
 r_{\mathrm{eff}}(\At)=\frac{R}{1+\mathrm{CV}_p^2}.
\end{equation}
\end{enumerate}
Moreover, $r_{\mathrm{eff}}$ is invariant under a common nonzero
rescaling of the $\lambda_r$, sign changes, and permutations of
the orthogonal components.
\end{Prop}
\begin{proof}
Since every $p_r>0$ and $\sum_rp_r=1$, we have
$\sum_rp_r^2\le1$, with equality exactly when $R=1$.
Cauchy--Schwarz gives
$1=(\sum_rp_r)^2\le R\sum_rp_r^2$;
equality holds exactly when $p_1=\cdots=p_R=1/R$,
equivalently when all $|\lambda_r|$ are equal.
This proves the first claim.

For a fixed mode $j$, let $U_j$ have columns $\uv_{jr}$ and let
$V_j$ have columns
$\bigotimes_{\ell\ne j}\uv_{\ell r}$ (in the unfolding
order). Both sets of columns are orthonormal by \eqref{orthogonal-cp}.
Thus $\Am_{(j)}=U_j\operatorname{diag}(\lambda_r)V_j^T$ has
nonzero singular values $|\lambda_r|$, and $G_j$ has nonzero
eigenvalues $\lambda_r^2=\|\At\|_F^2p_r$. Hence
$\operatorname{tr}G_j=\|\At\|_F^2$ and
$\operatorname{tr}(G_j^2)=\|\At\|_F^4\sum_rp_r^2$, proving
\eqref{effective-rank-unfolding}.

Finally, $\sum_{r=1}^R p_r=1$ and $\bar p=1/R$ give
\[ \mathrm{CV}_p^2
 =\frac{\frac{1}{R}\sum_{r=1}^R(p_r-\bar p)^2}{\bar p^2}
 =R\sum_{r=1}^R p_r^2-1
 =\frac{R}{r_{\mathrm{eff}}(\At)}-1, \]
which proves \eqref{effective-cv-definition} and
\eqref{effective-cv-identity}. The invariances
follow directly from \eqref{effective-rank}.
\end{proof}

Consequently, $r_{\mathrm{eff}}$ approaches one as a single
$p_r$ approaches one and the others approach zero, although
$r_{\mathrm{eff}}>1$ for every $R>1$ under \eqref{orthogonal-cp}.
It equals $R$ for equal component energies. Formula
\eqref{effective-cv-identity} makes clear why the number $R$ alone
cannot predict the variance: highly unequal $p_r$ can give a
small effective rank even when $R$ is large.

\begin{Thm}[Exact Gaussian variance for orthogonal CP tensors]\label{thm:ocp-gaussian}
  Let $\At$ satisfy \eqref{orthogonal-cp}, and let
\begin{equation*}
 Y=\left\|(\cdots(\At\times_d\xv_d)\cdots)\times_2\xv_2\right\|_2^2
\end{equation*}
with mutually independent standard $\mathds F$-Gaussian probes. Then
\begin{equation}\label{ocp-gaussian-var}
 \operatorname{Var}(Y)
 =\left(\varrho(2,\mathds F)^{d-1}-1\right)
 \sum_{r=1}^{R}\lambda_r^4.
\end{equation}
Consequently,
\begin{equation}\label{ocp-gaussian-relative}
 \frac{\operatorname{Var}(\mathrm{Est}_k^G)}{\|\At\|_F^4}
 =\frac{\varrho(2,\mathds F)^{d-1}-1}
 {k\,r_{\mathrm{eff}}(\At)}.
\end{equation}
\end{Thm}
\begin{proof}
Set
\begin{equation*}
  g_{jr}:=\langle\uv_{jr},\xv_j\rangle,
  \qquad j=2,\ldots,d,\quad r=1,\ldots,R.
\end{equation*}
After contraction along modes $2,\ldots,d$, the resulting vector is
\begin{equation*}
  (\cdots(\At\times_d\xv_d)\cdots)\times_2\xv_2
  =
  \sum_{r=1}^{R}
  \lambda_r\uv_{1r}g_{dr}\cdots g_{2r}.
\end{equation*}
The orthogonality of the first-mode factors,
$\uv_{1r}^{T}\uv_{1s}=\delta_{rs}$, eliminates all cross terms when
the squared norm is taken. Therefore,
\begin{equation*}
  Y
  =
  \sum_{r=1}^{R}\lambda_r^2Y_r,
  \qquad
  Y_r:=\prod_{j=2}^{d}|g_{jr}|^2.
\end{equation*}

The orthogonality in the remaining modes is used to determine the joint
distribution of the variables $Y_r$. For each fixed
$j=2,\ldots,d$, the vectors
$\uv_{j1},\ldots,\uv_{jR}$ are orthonormal. By the rotational
invariance of the standard Gaussian distribution,
\begin{equation*}
  g_{j1},\ldots,g_{jR}
\end{equation*}
are independent standard Gaussian variables over $\mathds F$.
Moreover, the probe vectors $\xv_2,\ldots,\xv_d$ are mutually
independent. It follows that all the variables
$\{g_{jr}:j=2,\ldots,d,\ r=1,\ldots,R\}$ are mutually independent,
and hence $Y_1,\ldots,Y_R$ are independent. Then, by Theorem \ref{thm:rank-one-exact}, the random variables $Y_1,\ldots,Y_R$ satisfy
\begin{equation*}
 \mathbb EY_r=1,
 \qquad
 \mathbb EY_r^2=\varrho(2,\mathds F)^{d-1}.
\end{equation*}
It follows that
\begin{equation*}
 \operatorname{Var}(Y)
 =\sum_{r=1}^{R}\lambda_r^4\operatorname{Var}(Y_r)
 =\left(\varrho(2,\mathds F)^{d-1}-1\right)
 \sum_{r=1}^{R}\lambda_r^4.
\end{equation*}

Finally, orthogonality of the CP components also gives $\|\At\|_F^2=\sum_{r=1}^{R}\lambda_r^2$.
Dividing the preceding variance formula by $\|\At\|_F^4$, using
the definition of $r_{\mathrm{eff}}(\At)$, and averaging $k$
independent samples yield
\begin{equation*}
  \frac{\operatorname{Var}(\mathrm{Est}_k^G)}{\|\At\|_F^4}
  =
  \frac{\varrho(2,\mathds F)^{d-1}-1}
  {k\,r_{\mathrm{eff}}(\At)}.
\end{equation*}
This proves \eqref{ocp-gaussian-relative}.
\end{proof}

Identity \eqref{ocp-gaussian-relative} gives the exact
variance reduction for an orthogonal CP tensor in terms of
$r_{\mathrm{eff}}(\At)$. The real and complex cases \cite{MA26} are consistent
with the Kronecker trace-estimation viewpoint discussed in the
introduction; the formula also covers quaternion probes. In
particular, evenly spread component energies reduce the variance
relative to the rank-one worst case.

\begin{Cor}[Effective-rank Chebyshev bound]\label{cor:effective-rank-bounds}
Let $\At\ne0$ satisfy \eqref{orthogonal-cp}. For every $\varepsilon>0$,
\begin{equation}\label{effective-rank-chebyshev}
 \mathbb P\left\{\frac{|\mathrm{Est}_k^G-\|\At\|_F^2|}{\|\At\|_F^2}>\varepsilon\right\}
 \le \frac{\varrho(2,\mathds F)^{d-1}-1}{k\,r_{\mathrm{eff}}(\At)\varepsilon^2}.
\end{equation}
\end{Cor}
\begin{proof}
Apply Chebyshev's inequality to the exact variance
\eqref{ocp-gaussian-relative}.
\end{proof}


\begin{Cor}[Effective rank under component truncation]\label{cor:ocp-truncation}
Under \eqref{orthogonal-cp}, let $S\subseteq\{1,\ldots,R\}$ and
$\At_S=\sum_{r\in S}\lambda_r\uv_{1r}\circ\cdots\circ\uv_{dr}\ne0$.
Then its relative variance is
\[
 \frac{\operatorname{Var}(\mathrm{Est}_k^G(\At_S))}{\|\At_S\|_F^4}
 =\frac{\varrho(2,\mathds F)^{d-1}-1}{k}
 \left(\frac{(\sum_{r\in S}\lambda_r^2)^2}{\sum_{r\in S}\lambda_r^4}\right)^{-1}.
\]
In particular, with $R$ equal weights this is proportional to $1/R$;
removing orthogonal components of negligible squared weight does not
materially change the variance coefficient.
\end{Cor}

We illustrate Theorem~\ref{thm:ocp-gaussian} with an order-four
orthogonal CP tensor whose four mode dimensions are $1000$. Set
$R=10$. For each $\alpha$ and each mode $j$, the MATLAB program forms
$U_j=\operatorname{orth}(G_j)$ from an independent
$G_j\in\REAL^{1000\times10}$ with i.i.d. standard Gaussian entries.
The columns $\uv_{j1},\ldots,\uv_{jR}$ of $U_j$ are consequently
orthonormal, so the numerical tensor satisfies \eqref{orthogonal-cp}. Choose
$\lambda_1=1$ and $\lambda_2=\cdots=\lambda_{10}=\alpha$ for
$\alpha\in\{0.1,0.2,\ldots,1\}$. Thus all coefficients are nonzero,
$\|\At\|_F^2=1+9\alpha^2$, and
\begin{equation}\label{effective-rank-example-alpha}
 r_{\mathrm{eff}}(\At)=\frac{(1+9\alpha^2)^2}{1+9\alpha^4},
\end{equation}
which grows from approximately $1.187$ to $10$ over this grid.

For each $\alpha$, we use $200$ common real Gaussian vectors per
contracted mode. Pairing adjacent real vectors yields $100$ complex
probes, each scaled by $1/\sqrt2$; grouping them in fours yields
$50$ quaternion probes, each scaled by $1/2$. The three estimators
therefore use the same $200$ underlying real vectors per mode, although
they average different numbers of samples. We repeat each experiment
$M=500$ times independently and report the sample variance of the
estimates divided by $\|\At\|_F^4$. Theorem~\ref{thm:ocp-gaussian}
predicts the three relative-variance curves
\begin{equation}\label{effective-rank-example-curves}
 \frac{\operatorname{Var}^{\REAL}(\mathrm{Est}_{200}^G)}{\|\At\|_F^4}=\frac{52}{400\,r_{\mathrm{eff}}},\qquad
 \frac{\operatorname{Var}^{\COPX}(\mathrm{Est}_{100}^G)}{\|\At\|_F^4}=\frac{28}{400\,r_{\mathrm{eff}}},\qquad
 \frac{\operatorname{Var}^{\QUAN}(\mathrm{Est}_{50}^G)}{\|\At\|_F^4}=\frac{19}{400\,r_{\mathrm{eff}}}.
\end{equation}

\begin{figure}[tb]
\centering
\includegraphics[width=0.82\textwidth]{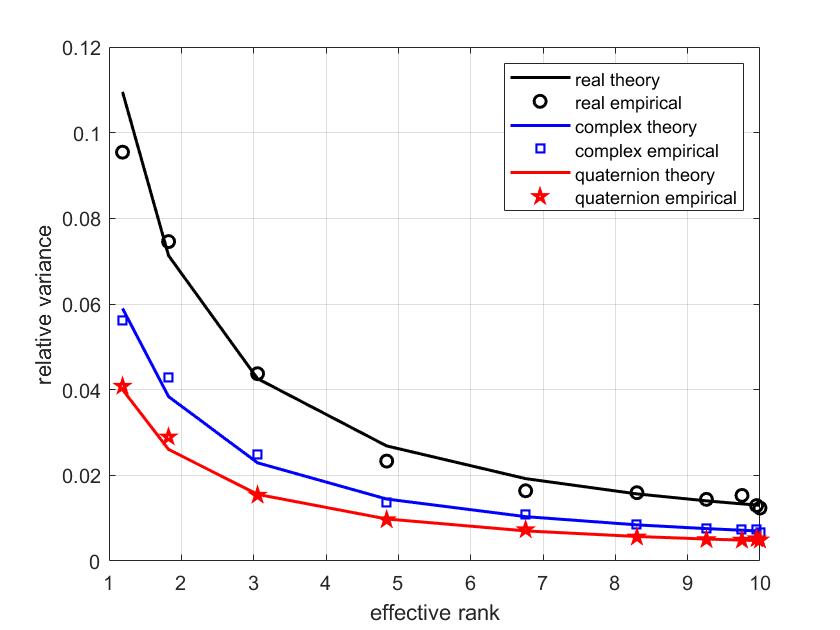}
\caption{Relative variance versus effective rank for the
 4th order orthogonal CP tensors with $R=10$ and
orthonormal factor columns in every mode. Solid curves are the exact
predictions in \eqref{effective-rank-example-curves}; markers are
sample variances from $500$ independent repetitions. The real,
complex, and quaternion estimators share the same underlying real
Gaussian draws in each repetition.}
\label{fig:effective-rank-variance}
\end{figure}

Figure~\ref{fig:effective-rank-variance} shows an approximately
$1/r_{\mathrm{eff}}$ decrease for each field, as predicted by
Theorem~\ref{thm:ocp-gaussian}. At $r_{\mathrm{eff}}=10$, the exact
relative variances are $0.013$, $0.007$, and $0.00475$ for the
real, complex, and quaternion estimators, respectively. The empirical
markers track the corresponding curves, with finite-repetition
fluctuations. Their ordering also agrees with
\eqref{effective-rank-example-curves}: under this fixed real-randomness
budget, quaternion probes have the smallest variance. These results
provide numerical support for the exact effective-rank dependence
and field constants in Theorem~\ref{thm:ocp-gaussian}.

\section{General tensors: sharp moments and concentration}\label{sec:general-main}

We now establish uniform guarantees without CP or rank assumptions.

\begin{Lem}[Gaussian matrix moment bound]\label{Lem:NR1}
  Let $\xv\sim\mathsf N_{\mathds F}(\mathbf0,\Id_n)$ and let $\Qm\in\mathds F^{m\times n}$ be fixed. For every integer $p\ge1$,
\begin{equation}\label{gaussian-matrix-moment}
 \mathbb E\|\Qm\xv\|_2^{2p}
 \le\varrho(p,\mathds F)\|\Qm\|_F^{2p}.
\end{equation}
Equality holds when $p=1$ or $\operatorname{rank}(\Qm)\le1$. If $p>1$ and $\operatorname{rank}(\Qm)>1$, the inequality is strict.
\end{Lem}
\begin{proof}
  The assertion is immediate for $\Qm=\mathbf0$. Otherwise, let $\sigma_1,\ldots,\sigma_r$ be its nonzero singular values and set
\begin{equation*}
 w_\ell=\frac{\sigma_\ell^2}{\|\Qm\|_F^2},
 \qquad
 Z_\ell=|\langle\vv_\ell,\xv\rangle|^2.
\end{equation*}
Gaussian unitary invariance, including the quaternion case
explained after Lemma~\ref{Lem:3.1}, shows that the right singular
factor maps $\xv$ to a standard $\mathds F$-Gaussian vector.
Hence the $Z_\ell$ are independent copies of $|g|^2$, and
\begin{equation*}
 \|\Qm\xv\|_2^2
 =\|\Qm\|_F^2\sum_{\ell=1}^r w_\ell Z_\ell,
 \qquad \sum_{\ell=1}^r w_\ell=1.
\end{equation*}
The convexity of $t\mapsto t^p$ $(t\ge0)$ and Lemma~\ref{lem:scalar-gamma} give
\begin{equation*}
 \mathbb E\|\Qm\xv\|_2^{2p}
 \le\|\Qm\|_F^{2p}
 \sum_{\ell=1}^r w_\ell\mathbb EZ_\ell^p
 =\varrho(p,\mathds F)\|\Qm\|_F^{2p}.
\end{equation*}
The equality statements follow from linearity when $p=1$ and strict convexity when $p>1$.
\end{proof}

\begin{Thm}[Uniform variance bound]\label{Thm:Var}
  Let $\At\in\REAL^{n_1\times\cdots\times n_d}$ and let $\mathrm{Est}_k^G$ be defined by \eqref{tensor-estimator}. Then
\begin{equation}\label{uniform-variance-bound}
 \operatorname{Var}(\mathrm{Est}_k^G)
 \le\frac{\varrho(2,\mathds F)^{d-1}-1}{k}\|\At\|_F^4,
\end{equation}
where $\varrho(2,\REAL)=3$, $\varrho(2,\COPX)=2$, and $\varrho(2,\QUAN)=3/2$.
\end{Thm}
\begin{proof}
  It suffices to consider one sample. We use the real mode-$1$
  unfolding and retain the ordered quaternion products specified
  after \eqref{tensor-estimator}. Let
\begin{equation*}
 Y=\|\Am_{(1)}\zv\|_2^2,
 \qquad
 \zv=\xv_d\otimes\cdots\otimes\xv_2,
\end{equation*}
and write a real compact singular value decomposition
$\Am_{(1)}=\Um\operatorname{diag}(\sigma_1,\ldots,\sigma_r)[\vv_1,\dots,\vv_r]^T$.
With $\lambda_\ell=\sigma_\ell^2/\|\At\|_F^2$,
\begin{equation*}
 \frac{Y}{\|\At\|_F^2}
 =\sum_{\ell=1}^r\lambda_\ell
 |\langle\vv_\ell,\zv\rangle|^2,
 \qquad \sum_{\ell=1}^r\lambda_\ell=1.
\end{equation*}
Since the vectors $\vv_\ell$ have real entries, reshape each
as an order-$(d-1)$ real tensor. Condition first on the probes in
modes $3,\ldots,d$ and apply Lemma~\ref{Lem:NR1} to the remaining
mode-$2$ linear contraction. Repeat for modes $3,\ldots,d$ in order.
No quaternion factors are interchanged. This gives
\begin{equation}\label{kronecker-fourth-bound}
 \mathbb E|\langle\vv_\ell,\zv\rangle|^4
 \le\varrho(2,\mathds F)^{d-1}.
\end{equation}
Jensen's inequality and Proposition~\ref{Thm:unbiased} therefore yield
\begin{equation*}
 \mathbb EY^2
 \le\|\At\|_F^4\sum_{\ell=1}^r\lambda_\ell
 \mathbb E|\langle\vv_\ell,\zv\rangle|^4
 \le\varrho(2,\mathds F)^{d-1}\|\At\|_F^4.
\end{equation*}
Thus,
\begin{equation*}
 \operatorname{Var}(Y)
 \le\bigl(\varrho(2,\mathds F)^{d-1}-1\bigr)\|\At\|_F^4.
\end{equation*}
Averaging $k$ independent copies divides the variance by $k$.
\end{proof}

Real Kronecker-probe bounds are studied in \cite{BK21}, and
real and complex Kronecker trace-estimation variances in
\cite{MA26}. The bound here is stated for the common normalization
of all three algebras and includes quaternion Gaussian probes.
The rank-one result identifies a tensor attaining each uniform
constant.

\begin{Cor}[Sharp Gaussian moments and variance gaps]
\label{cor:sharpness}\label{cor:general-moments}
Let $\At\in\REAL^{n_1\times\cdots\times n_d}$ and let $Y$ be one sample
from \eqref{tensor-estimator}. For every
integer $p\ge1$,
\begin{equation}\label{general-moments}
 \mathbb E Y^p\le\varrho(p,\mathds F)^{d-1}\|\At\|_F^{2p}.
\end{equation}
Equality holds for every nonzero rank-one tensor. 
Consequently, the sharp worst-case relative-variance factors are
$3^{d-1}-1$, $2^{d-1}-1$, and $(3/2)^{d-1}-1$ over
$\REAL$, $\COPX$, and $\QUAN$, respectively. 
\end{Cor}

\begin{proof}
The claim is immediate when $\At=0$. For $\At\ne0$, use the real singular vectors and weights
$\lambda_\ell$ from the proof of Theorem~\ref{Thm:Var}.
Convexity gives, pointwise,
\[
 \left(\frac{Y}{\|\At\|_F^2}\right)^p
 =\left(\sum_{\ell=1}^r\lambda_\ell
       |\langle\vv_\ell,\zv\rangle|^2\right)^p
 \le\sum_{\ell=1}^r\lambda_\ell
       |\langle\vv_\ell,\zv\rangle|^{2p}.
\]
For each real $\vv_\ell$, condition successively on the independent
probes and apply Lemma~\ref{Lem:NR1} with exponent $p$ in every
contracted mode. Hence
$\mathbb E|\langle\vv_\ell,\zv\rangle|^{2p}
\le\varrho(p,\mathds F)^{d-1}\|\vv_\ell\|_2^{2p}
=\varrho(p,\mathds F)^{d-1}$, proving \eqref{general-moments}.
For rank one, \eqref{rank-one-moments} gives equality for every $p$.
At $p=2$, subtracting $(\mathbb EY)^2=\|\At\|_F^4$ yields the displayed
worst-case variance factors, attained by rank-one tensors.
\end{proof}

\begin{Cor}[General-tensor two-sided Chebyshev bound]
\label{cor:tail_bound}
Let $\At\ne0$. For every $\varepsilon>0$,
\[
 \mathbb P\left\{\frac{\left|\mathrm{Est}_k^G-\|\At\|_F^2\right|}{\|\At\|_F^2}>\varepsilon\right\}
 \le\frac{\varrho(2,\mathds F)^{d-1}-1}{k\varepsilon^2}.
\]
\end{Cor}

The preceding Chebyshev bound gives a two-sided failure probability
of order $1/k$, but does not by itself yield exponential confidence
for the ordinary sample mean. This distinction matters for
multilinear Gaussian probes: already for a rank-one tensor,
\eqref{rank-one-distribution} expresses one sample as a product of
$d-1$ independent squared Gaussian magnitudes. When $d\ge3$,
its positive moment generating function is infinite, as also
suggested by the heavy-tail behavior of Gaussian chaoses
\cite{La06}. Thus a direct two-sided Chernoff bound for the
unmodified mean is unavailable in general \cite{BK21}. A classical way to
amplify a constant-probability Chebyshev estimate is the
median-of-means construction \cite{BubeckCesaBianchiLugosi2013}.
We apply this standard principle to the tensor estimator using the
variance bound of Theorem~\ref{Thm:Var}.

\begin{Thm}[Two-sided exponential confidence for general tensors]
\label{thm:median-blocks}
Let $\At\ne0$ be any real order-$d$ tensor, and set
$\mu=\|\At\|_F^2$. Let $Y$ denote one independent sample from
\eqref{tensor-estimator}, and suppose that a known constant
$V\ge\operatorname{Var}(Y)/\mu^2$ is available. Given
$\varepsilon>0$, set
\[
 b=\max\{1,\lceil4V/\varepsilon^2\rceil\}
\]
and choose an odd integer $t\ge1$. Draw $bt$ independent copies
$Y_{j\ell}$ of $Y$, form the block means
\[
 B_j=\frac1b\sum_{\ell=1}^{b}Y_{j\ell},
 \qquad j=1,\ldots,t,
\]
and let $\widehat\mu_{b,t}=\operatorname{median}(B_1,\ldots,B_t)$.
Then
\begin{equation}\label{median-block-tail}
 \mathbb P\{|\widehat\mu_{b,t}-\mu|>\varepsilon\mu\}
 \le e^{-t/8}.
\end{equation}
For every tensor, one may take
$V=\varrho(2,\mathds F)^{d-1}-1$ by
Theorem~\ref{Thm:Var}. For an orthogonal CP tensor,
Theorem~\ref{thm:ocp-gaussian} permits the smaller value
$V=[\varrho(2,\mathds F)^{d-1}-1]/r_{\mathrm{eff}}(\At)$
when its effective rank is known.
\end{Thm}
\begin{proof}
Since the samples are independent and identically distributed,
each block mean satisfies
\[
 \mathbb E B_j=\mu,
 \qquad
 \operatorname{Var}(B_j)=\frac{\operatorname{Var}(Y)}{b}
 \le\frac{V\mu^2}{b}.
\]
Therefore, Chebyshev's inequality (Corollary \ref{cor:tail_bound}) and the choice of $b$ give
\begin{equation}\label{median-block-bad-prob}
 q_j:=\mathbb P\{|B_j-\mu|>\varepsilon\mu\}
 \le\frac{V}{b\varepsilon^2}\le\frac14,
 \qquad j=1,\ldots,t.
\end{equation}
Let $I_j=\mathbf1_{\{|B_j-\mu|>\varepsilon\mu\}}$.
The blocks are disjoint and are formed from independent samples,
so the Bernoulli variables $I_1,\ldots,I_t$ are independent.
Moreover, $\mathbb E\sum_{j=1}^t I_j=\sum_jq_j\le t/4$.
If the median lies outside
$[\mu(1-\varepsilon),\mu(1+\varepsilon)]$, at least
$(t+1)/2$ block means lie outside this interval. Hence
\[
 \{|\widehat\mu_{b,t}-\mu|>\varepsilon\mu\}
 \subseteq\left\{\sum_{j=1}^t I_j\ge\frac{t+1}{2}\right\}
 \subseteq\left\{\sum_{j=1}^t(I_j-\mathbb EI_j)
                   \ge\frac{t}{4}\right\}.
\]
Hoeffding's inequality for independent random variables in
$[0,1]$ \cite{Hoeffding1963} now yields
\[
 \mathbb P\left\{\sum_{j=1}^t(I_j-\mathbb EI_j)
                   \ge\frac{t}{4}\right\}
 \le\exp\left(-\frac{2(t/4)^2}{t}\right)=e^{-t/8},
\]
which proves \eqref{median-block-tail}. The two stated choices of
$V$ follow from \eqref{uniform-variance-bound} (with $k=1$)
and \eqref{ocp-gaussian-relative} (with $k=1$), respectively.
\end{proof}

The following algorithm makes the confidence parameters explicit.
For a prescribed $\delta\in(0,1)$, choosing the smallest odd
integer $t\ge8\log(1/\delta)$ in Theorem~\ref{thm:median-blocks}
ensures failure probability at most $\delta$.

\begin{center}
\refstepcounter{algorithm}\label{alg:median-block-norm}
\begin{minipage}{0.94\textwidth}
\noindent\textbf{Algorithm \thealgorithm.} Block-median tensor Frobenius-norm estimation.
\begin{enumerate}
\item \textbf{Input:} a contraction oracle for $\At\ne0$, a probe field
$\mathds F\in\{\REAL,\COPX,\QUAN\}$, accuracy $\varepsilon>0$,
confidence level $\delta\in(0,1)$, and optionally a known valid
relative-variance bound $V$.
\item If no bound is supplied, set
$V\gets\varrho(2,\mathds F)^{d-1}-1$. Set
$b\gets\max\{1,\lceil4V/\varepsilon^2\rceil\}$ and let $t$ be the
smallest odd integer satisfying $t\ge8\log(1/\delta)$.
\item For $j=1,\ldots,t$, draw $b$ independent probe tuples
$(\xv_2^{j\ell},\ldots,\xv_d^{j\ell})$, evaluate
\[
 Y_{j\ell}=\left\|(\cdots(\At\times_d\xv_d^{j\ell})\cdots)
 \times_2\xv_2^{j\ell}\right\|_2^2,
 \quad \ell=1,\ldots,b,
\]
and store $B_j\gets \frac{1}{b}\sum_{\ell=1}^bY_{j\ell}$.
\item \textbf{Return:}
$\widehat\mu\gets\operatorname{median}(B_1,\ldots,B_t)$.
\end{enumerate}
\end{minipage}
\end{center}

Algorithm~\ref{alg:median-block-norm} uses $bt$ independent
contractions and stores only the $t$ block means. Its output
satisfies
$$\mathbb P\{|\widehat\mu-\|\At\|_F^2|>
\varepsilon\|\At\|_F^2\}\le\delta$$ by
Theorem~\ref{thm:median-blocks}. This guarantee concerns the
block median, rather than the unmodified average of all $bt$
samples.

\section{Application to an implicit tensor-kernel problem}\label{Sec:NumExp}

Consider interpolation from values $\bv=(f_1,\ldots,f_N)^T$ at
pairwise distinct scattered sites $x_1,\ldots,x_N\in(0,1)$.
In a generalized Mercer kernel formulation \cite{XY19,CFLY24},
the coefficient vector $\cv$ solves the order-$d$ tensor equation
\begin{equation*}
  \At\cv^{d-1}=\bv, \qquad \text{ with even }d.
\end{equation*}
The tensor has $N^d$
entries but is specified implicitly through an infinite series.
For a concrete sine-series kernel, set
\begin{equation}\label{kernel-series}
 \uv_n:=\bigl(\sin(n\pi x_i)\bigr)_{i=1}^N,\qquad
 \lambda_n:=\frac{2^{d/2}}{\pi^d n^d},\qquad
 \At=\sum_{n=1}^{\infty}\lambda_n\uv_n^{\circ d}.
\end{equation}
For $d=2$, this normalization gives the familiar kernel
$K(x,y)=\min(x,y)-xy$. For a finite computation, we retain the first $p$ terms,
\begin{equation*}
  \At_p=\sum_{n=1}^p\lambda_n\uv_n^{\circ d}
\end{equation*}
and use random probes to estimate $\|\At_p\|_F^2$ without forming $\At_p$. The example illustrates this use of the estimator; we do not study other algorithms for selecting $p$ or interpolating the data.

The factors $\uv_n$ are generally not orthogonal at scattered
sites. Indeed, the exact norm of a finite truncation contains
cross terms,
\begin{equation}\label{kernel-gram-norm}
 \|\At_p\|_F^2
 =\sum_{m=1}^{p}\sum_{n=1}^{p}\lambda_m\lambda_n
   (\uv_m^T\uv_n)^d.
\end{equation}
Consequently, the orthogonal-CP identity of
Theorem~\ref{thm:ocp-gaussian} does not apply to this example;
the uniform bound of Theorem~\ref{Thm:Var} does. Evaluating
\eqref{kernel-gram-norm} avoids forming $N^d$ entries, but involves
all pairs of retained terms. The random-probe method instead needs
only the contractions
\begin{equation}\label{kernel-contraction}
 \zv_p(\xv_2,\ldots,\xv_d)
 :=(\cdots(\At_p\times_d\xv_d)\cdots)\times_2\xv_2
 =\sum_{n=1}^{p}\lambda_n
   (\uv_n^T\xv_d)\cdots(\uv_n^T\xv_2)\uv_n.
\end{equation}
For any fixed $p$ and independent standard probes over
$\mathds F\in\{\REAL,\COPX,\QUAN\}$, the sample mean of
$\|\zv_p\|_2^2$ is unbiased for $\|\At_p\|_F^2$ and obeys
\begin{equation}\label{kernel-fixed-truncation-variance}
 \operatorname{Var}(\mathrm{Est}_k^G(\At_p))
 \le\frac{\varrho(2,\mathds F)^{d-1}-1}{k}
        \|\At_p\|_F^4.
\end{equation}
The same statement applies to each fixed block
$\At_{p+\ell}-\At_p$. These are direct applications of the
general-tensor results in Section~\ref{sec:general-main}.

In our experiment, $d=4$, the block length is $\ell=5$,
and $k=50$ probe tuples are reused as the series is traversed.
For each tuple, the contraction in \eqref{kernel-contraction} can
be updated by adding only the next $\ell$ summands. This yields
two estimates at each step: one for the current partial tensor
and one for the next block. Algorithm~\ref{Alg:RandomTen}
states the empirical stopping rule used in this example.
It uses the estimated norms to select a truncation in this experiment.

\begin{center}
\refstepcounter{algorithm}\label{Alg:RandomTen}
\begin{minipage}{0.94\textwidth}
\noindent\textbf{Algorithm \thealgorithm.} Probe-guided truncation of the kernel series.
\begin{enumerate}
\item \textbf{Input:} the series data in \eqref{kernel-series}, a probe field
$\mathds F$, sample count $k=50$, block length $\ell=5$, and
threshold $\tau=\sqrt{\mathrm{eps}}$, where $\mathrm{eps}$ is machine precision.
\item Draw $k$ independent tuples $(\xv_2^i,\ldots,\xv_d^i)$.
Set $p\gets0$ and $\zv_p^i\gets\mathbf0$ for $i=1,\ldots,k$.
\item For the next block, compute
\[
 \Delta\zv_p^i
 =\sum_{n=p+1}^{p+\ell}\lambda_n
   (\uv_n^T\xv_d^i)\cdots(\uv_n^T\xv_2^i)\uv_n,
 \qquad i=1,\ldots,k.
\]
Set $D_p\gets\bigl(\frac{1}{k}\sum_i\|\Delta\zv_p^i\|_2^2\bigr)^{1/2}$
and $N_p\gets\bigl(\frac{1}{k}\sum_i\|\zv_p^i\|_2^2\bigr)^{1/2}$.
\item If $D_p<\tau\max\{N_p,1\}$, return the implicit truncation
$\At_p=\sum_{n=1}^p\lambda_n\uv_n^{\circ d}$.
Otherwise set $\zv_{p+\ell}^i\gets\zv_p^i+\Delta\zv_p^i$,
$p\gets p+\ell$, and return to Step~3.
\end{enumerate}
\end{minipage}
\end{center}

The same probes evaluate the current and incremental contractions,
so no new probe tuples are required for successive blocks. The
stopping rule tests the next block, not the entire infinite tail;
it is therefore an empirical truncation criterion. The fixed-truncation variance bound alone does not certify a stopping decision made by repeatedly testing the same probes.

\begin{figure}[H]
\centering
\includegraphics[width=0.8\textwidth]{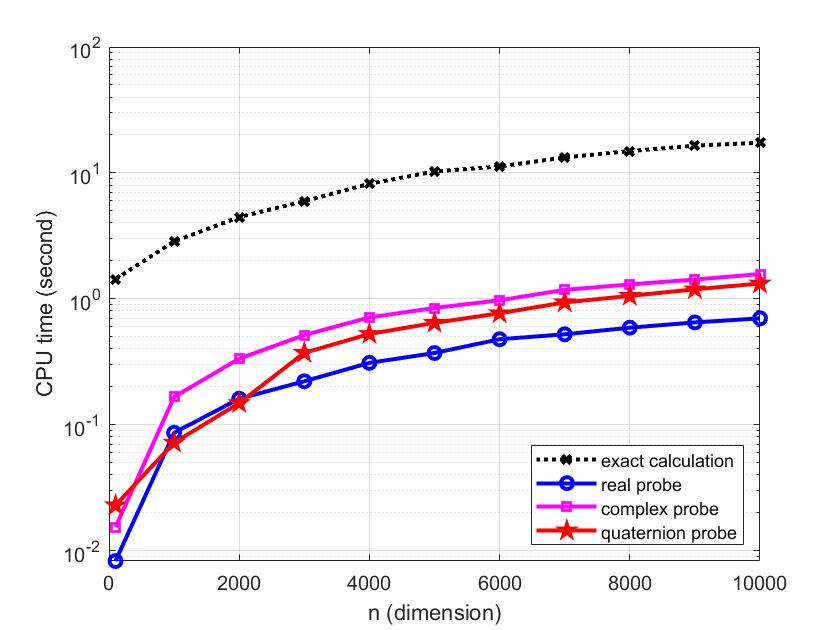}
\caption{CPU time for exact norm evaluation and probe-based norm estimation over $\REAL$, $\COPX$, and $\QUAN$ as the number $N$ of scattered sites increases. The parameters are $k=50$ and $\ell=5$.}
\label{RKBS-vDimen}
\end{figure}

Figure~\ref{RKBS-vDimen} reports the running times of this implementation.
For the dimensions shown, all three probe-based evaluations are faster
than the direct norm calculation in \eqref{kernel-gram-norm}. The real
probe is fastest in this experiment. These timings illustrate that the
implicit tensor can be handled through contractions; they do not compare
different truncation or interpolation methods.

\begin{figure}[H]
\centering
\begin{tabular}{cc}
 \includegraphics[width=0.3\textwidth]{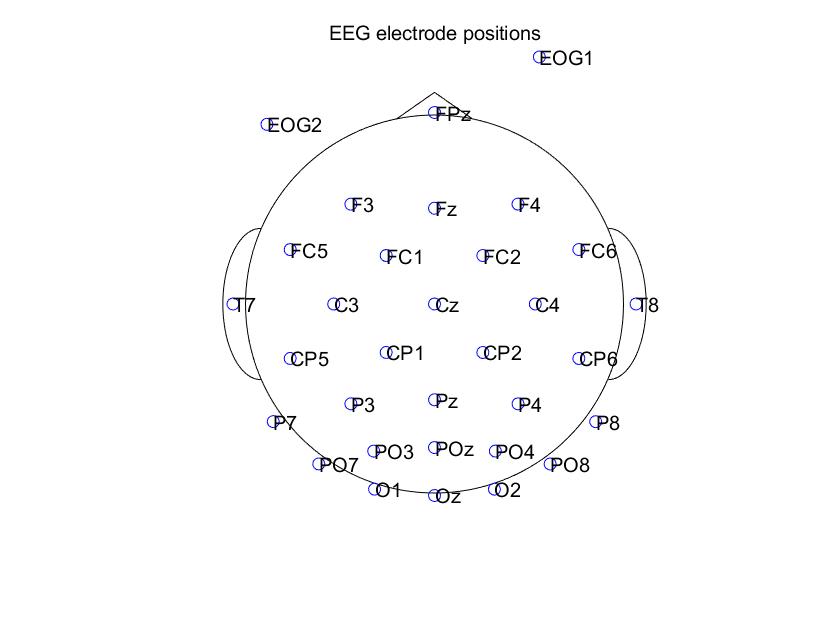} &
 \includegraphics[width=0.3\textwidth]{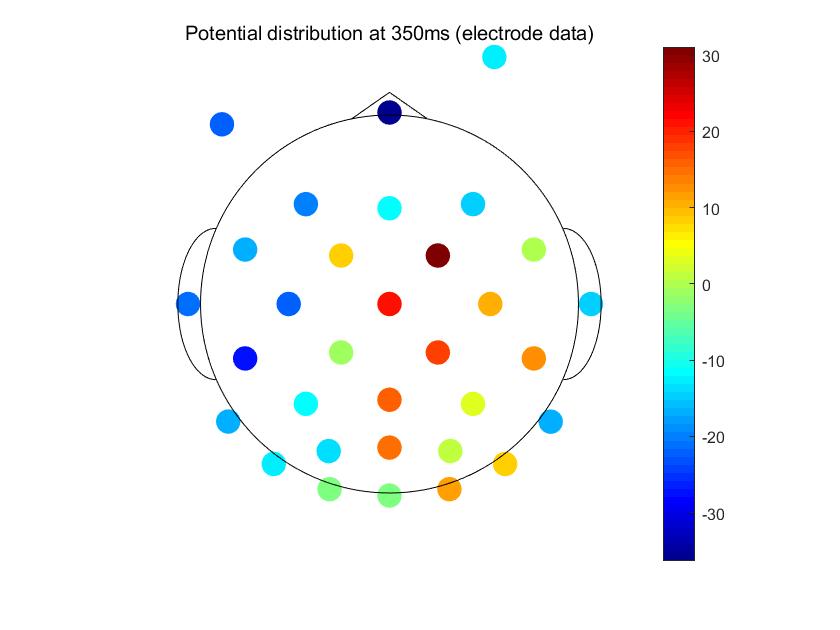} \\
 (a) & (b)
\end{tabular}
\caption{EEG example: (a) electrode locations and (b) measured
potentials at the selected post-stimulus time point.}
\label{EEG:setting}
\end{figure}

\begin{figure}[H]
\centering
\includegraphics[width=\textwidth]{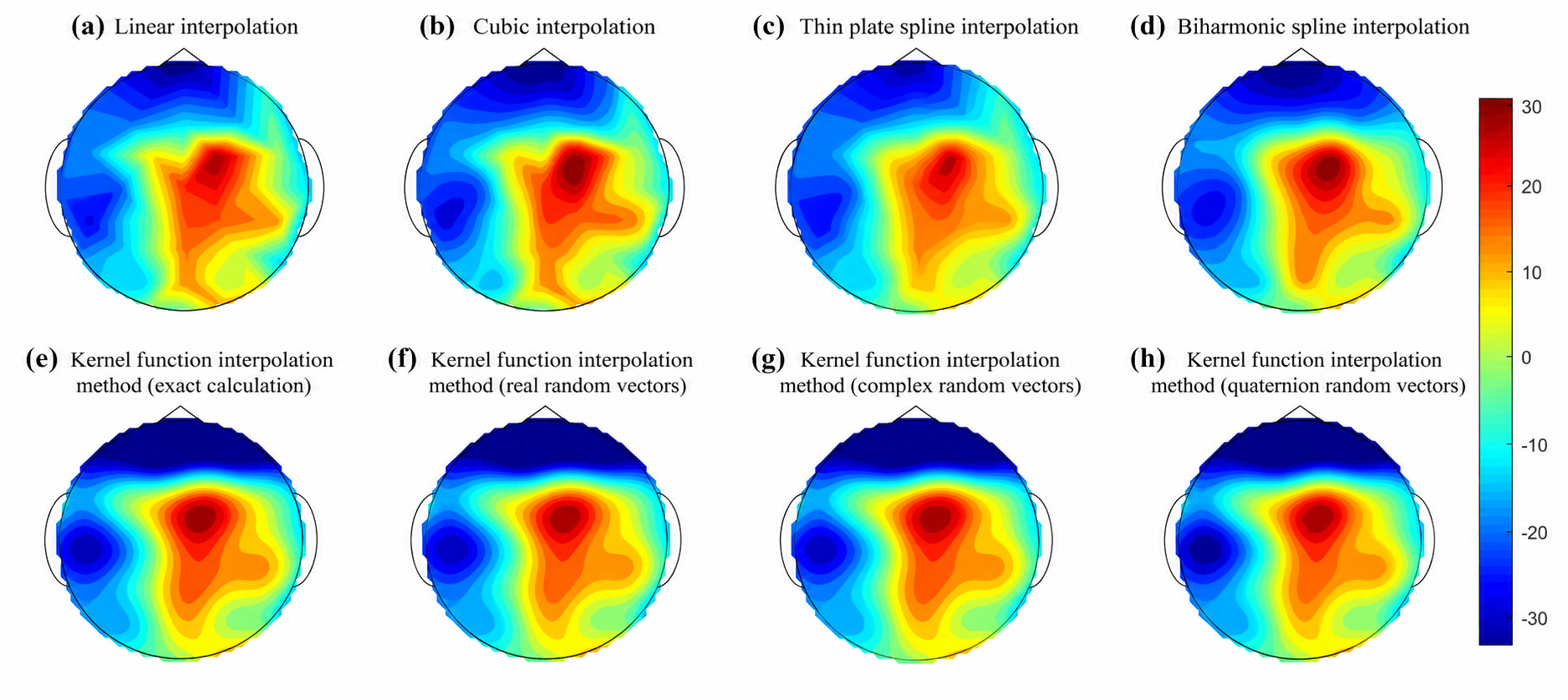}
\caption{Reconstructed scalp-potential maps: (a)--(d) classical
interpolation schemes; (e) tensor-kernel interpolation with an
exact finite-truncation norm; and (f)--(h) the same method using
real, complex, and quaternion norm probes, respectively.}
\label{EEG:reconstr}
\end{figure}

As a related spatial interpolation illustration, we use
scalp potentials recorded at $32$ EEG electrodes and reconstruct a
spatial potential map at one post-stimulus time point. The electrode
locations and measured values are shown in Figure~\ref{EEG:setting}.
The data are processed in MATLAB with EEGLAB \cite{DM04eeglab}.
In the displayed run, the tensor-kernel maps obtained with
exact and probe-based norm evaluations look similar in
Figure~\ref{EEG:reconstr}(e)--(h). These images illustrate the use
of probe-based norm estimates in this interpolation calculation.
We include the other maps for visual context; the images alone do
not measure reconstruction errors or compare interpolation methods.

\section{Conclusion}\label{Sec:Conclude}

We analyzed squared Frobenius-norm estimation for real tensors using independent real, complex, and quaternion Gaussian vectors in successive mode contractions. The exact rank-one distribution explains the sharp worst-case variance factors $3^{d-1}-1$, $2^{d-1}-1$, and $(3/2)^{d-1}-1$. Under equal underlying real randomness, modewise pooling gives quaternion probes a strict rank-one variance advantage over four independently averaged real probes when $d\ge3$. This is a variance comparison at a fixed budget of Gaussian draws.

For orthogonal CP tensors, the participation-ratio effective rank gives the exact variance reduction. For arbitrary real tensors, the moment bound and its rank-one equality establish sharp uniform constants. Chebyshev's inequality controls the ordinary sample mean, while the block-median estimator supplies two-sided exponential confidence from the same finite-variance information. The latter bound uses a standard median-of-means argument with the variance constants established here.

The orthogonal-CP and rank-one experiments are consistent with
the predicted variance formulas within their finite number of
repetitions. The infinite-series example illustrates norm monitoring
through contractions. Its repeated-probe stopping rule is empirical,
and the EEG maps provide a qualitative illustration. Future work
includes certified adaptive stopping, a specified spatial EEG kernel
with measured errors, and comparisons at matched arithmetic cost.

\bibliographystyle{plain}
\bibliography{myreferences}

@article{AT2011,
  author = {Haim Avron and Sivan Toledo},
  title = {Randomized Algorithms for Estimating the Trace of an Implicit Symmetric Positive Semi-Definite Matrix},
  journal = {Journal of the ACM},
  volume = {58},
  number = {2},
  pages = {8:1--8:17},
  year = {2011},
  doi = {10.1145/1944345.1944349}
}

@article{BK08,
  author = {Brett W. Bader and Tamara G. Kolda},
  title = {Efficient {MATLAB} Computations with Sparse and Factored Tensors},
  journal = {SIAM Journal on Scientific Computing},
  volume = {30},
  number = {1},
  pages = {205--231},
  year = {2008},
  doi = {10.1137/060676489}
}

@article{BK21,
  author = {Zvonimir Bujanovi{\'c} and Daniel Kressner},
  title = {Norm and Trace Estimation with Random Rank-One Vectors},
  journal = {SIAM Journal on Matrix Analysis and Applications},
  volume = {42},
  number = {1},
  pages = {202--223},
  year = {2021},
  doi = {10.1137/20M1331718}
}

@article{BubeckCesaBianchiLugosi2013,
  author = {S{\'e}bastien Bubeck and Nicol{\`o} Cesa-Bianchi and G{\'a}bor Lugosi},
  title = {Bandits with Heavy Tail},
  journal = {IEEE Transactions on Information Theory},
  volume = {59},
  number = {11},
  pages = {7711--7717},
  year = {2013},
  doi = {10.1109/TIT.2013.2277869}
}

@book{CB24,
  author = {George Casella and Roger L. Berger},
  title = {Statistical Inference},
  edition = {Second},
  publisher = {Chapman and Hall/CRC},
  year = {2024},
  isbn = {9781032593036}
}

@article{CFLY24,
  author = {Yannan Chen and Kaidong Fu and Can Li and Qi Ye},
  title = {A Scalable Optimization Approach for the Multilinear System Arising from Scattered Data Interpolation},
  journal = {CSIAM Transactions on Applied Mathematics},
  volume = {5},
  number = {3},
  pages = {590--614},
  year = {2024},
  doi = {10.4208/csiam-am.SO-2023-0045}
}

@article{DM04eeglab,
  author = {Arnaud Delorme and Scott Makeig},
  title = {{EEGLAB}: An Open Source Toolbox for Analysis of Single-Trial {EEG} Dynamics Including Independent Component Analysis},
  journal = {Journal of Neuroscience Methods},
  volume = {134},
  number = {1},
  pages = {9--21},
  year = {2004},
  doi = {10.1016/j.jneumeth.2003.10.009}
}

@article{H1990,
  author = {Michael F. Hutchinson},
  title = {A Stochastic Estimator of the Trace of the Influence Matrix for Laplacian Smoothing Splines},
  journal = {Communications in Statistics---Simulation and Computation},
  volume = {19},
  number = {2},
  pages = {433--450},
  year = {1990},
  doi = {10.1080/03610919008812866}
}

@article{Hamilton1844,
  author = {William Rowan Hamilton},
  title = {On Quaternions; or on a New System of Imaginaries in Algebra},
  journal = {The London, Edinburgh, and Dublin Philosophical Magazine and Journal of Science},
  series = {3},
  volume = {25},
  number = {163},
  pages = {10--13},
  year = {1844},
  doi = {10.1080/14786444408644923}
}

@article{Hoeffding1963,
  author = {Wassily Hoeffding},
  title = {Probability Inequalities for Sums of Bounded Random Variables},
  journal = {Journal of the American Statistical Association},
  volume = {58},
  number = {301},
  pages = {13--30},
  year = {1963},
  doi = {10.1080/01621459.1963.10500830}
}

@article{KoldaBader09,
  author = {Tamara G. Kolda and Brett W. Bader},
  title = {Tensor Decompositions and Applications},
  journal = {SIAM Review},
  volume = {51},
  number = {3},
  pages = {455--500},
  year = {2009},
  doi = {10.1137/07070111X}
}

@article{La06,
  author = {Rafa{\l} Lata{\l}a},
  title = {Estimates of Moments and Tails of {G}aussian Chaoses},
  journal = {The Annals of Probability},
  volume = {34},
  number = {6},
  pages = {2315--2331},
  year = {2006},
  doi = {10.1214/009117906000000421}
}

@article{MA26,
  author = {Raphael A. Meyer and Haim Avron},
  title = {{H}utchinson's Estimator Is Bad at {K}ronecker-Trace-Estimation},
  journal = {SIAM Journal on Matrix Analysis and Applications},
  volume = {47},
  number = {1},
  pages = {353--387},
  year = {2026},
  doi = {10.1137/24M1720895}
}

@article{RecanatesiEtAl2022,
  author = {Stefano Recanatesi and Serena Bradde and Vijay Balasubramanian and Nicholas A. Steinmetz and Eric Shea-Brown},
  title = {A Scale-Dependent Measure of System Dimensionality},
  journal = {Patterns},
  volume = {3},
  number = {8},
  pages = {100555},
  year = {2022},
  doi = {10.1016/j.patter.2022.100555}
}

@book{TensBK,
  author = {Grey Ballard and Tamara G. Kolda},
  title = {Tensor Decompositions for Data Science},
  publisher = {Cambridge University Press},
  year = {2025},
  doi = {10.1017/9781009471664}
}

@article{TookMandic2011,
  author = {Clive Cheong Took and Danilo P. Mandic},
  title = {Augmented Second-Order Statistics of Quaternion Random Signals},
  journal = {Signal Processing},
  volume = {91},
  number = {2},
  pages = {214--224},
  year = {2011},
  doi = {10.1016/j.sigpro.2010.06.024}
}

@article{XY19,
  author = {Yuesheng Xu and Qi Ye},
  title = {Generalized {M}ercer Kernels and Reproducing Kernel {B}anach Spaces},
  journal = {Memoirs of the American Mathematical Society},
  volume = {258},
  number = {1243},
  pages = {1--122},
  year = {2019},
  doi = {10.1090/memo/1243}
}

@article{Oseledets11,
  author = {Ivan V. Oseledets},
  title = {Tensor-Train Decomposition},
  journal = {SIAM Journal on Scientific Computing},
  volume = {33},
  number = {5},
  pages = {2295--2317},
  year = {2011},
  doi = {10.1137/090752286}
}

@article{BachmayrDahmen15,
  author = {Markus Bachmayr and Wolfgang Dahmen},
  title = {Adaptive Near-Optimal Rank Tensor Approximation for High-Dimensional Operator Equations},
  journal = {Foundations of Computational Mathematics},
  volume = {15},
  number = {4},
  pages = {839--898},
  year = {2015},
  doi = {10.1007/s10208-013-9187-3}
}

@article{AnandkumarEtAl14,
  author = {Animashree Anandkumar and Rong Ge and Daniel Hsu and Sham M. Kakade and Matus Telgarsky},
  title = {Tensor Decompositions for Learning Latent Variable Models},
  journal = {Journal of Machine Learning Research},
  volume = {15},
  number = {80},
  pages = {2773--2832},
  year = {2014},
  url = {https://jmlr.org/papers/v15/anandkumar14b.html}
}

@article{BaydinEtAl18,
  author = {Atilim Gunes Baydin and Barak A. Pearlmutter and Alexey Andreyevich Radul and Jeffrey Mark Siskind},
  title = {Automatic Differentiation in Machine Learning: A Survey},
  journal = {Journal of Machine Learning Research},
  volume = {18},
  number = {153},
  pages = {1--43},
  year = {2018},
  url = {https://jmlr.org/papers/v18/17-468.html}
}

@article{CK22,
  author = {Alice Cortinovis and Daniel Kressner},
  title = {On Randomized Trace Estimates for Indefinite Matrices with an Application to Determinants},
  journal = {Foundations of Computational Mathematics},
  volume = {22},
  number = {3},
  pages = {875--903},
  year = {2022},
  doi = {10.1007/s10208-021-09525-9}
}

@article{LQT20,
  author = {Ziyan Luo and Liqun Qi and Philippe L. Toint},
  title = {Tensor {B}ernstein Concentration Inequalities with an Application to Sample Estimators for High-Order Moments},
  journal = {Frontiers of Mathematics in China},
  volume = {15},
  number = {2},
  pages = {367--384},
  year = {2020},
  doi = {10.1007/s11464-020-0830-4}
}

@article{CheWeiYan2026,
  author = {Maolin Che and Yimin Wei and Hong Yan},
  title = {Efficient Randomized Algorithms for Computing an Approximation of the Tensor Train Decomposition},
  journal = {Journal of Scientific Computing},
  volume = {107},
  number = {1},
  pages = {Article 2},
  year = {2026},
  doi = {10.1007/s10915-026-03208-9}
}

@article{SaibabaMiedlar2025,
  author = {Arvind K. Saibaba and Agnieszka Mi\k{e}dlar},
  title = {Randomized Low-Rank Approximations beyond Gaussian Random Matrices},
  journal = {SIAM Journal on Mathematics of Data Science},
  volume = {7},
  number = {1},
  pages = {136--162},
  year = {2025},
  doi = {10.1137/23M1593255}
}

@article{ChangCuiHanWang2026,
  author = {Jingya Chang and Xiaofei Cui and Deren Han and Zhenyou Wang},
  title = {Hyperspectral Super-Resolution Via Tensor {CP} Decomposition with Total Variation Regularization},
  journal = {Journal of Optimization Theory and Applications},
  volume = {209},
  number = {3},
  pages = {Article 80},
  year = {2026},
  doi = {10.1007/s10957-026-03020-3}
}

@article{ChenQi2026,
  author = {Yannan Chen and Liqun Qi},
  title = {A Direct Method for Computing the Complex-Valued Triple Decomposition of Third-Order Tensors},
  journal = {SIAM Journal on Scientific Computing},
  volume = {48},
  number = {2},
  pages = {A425--A447},
  year = {2026},
  doi = {10.1137/24M1660164}
}

\end{document}